\documentclass[reqno]{amsart}
\advance \oddsidemargin -0.6in 
\advance \evensidemargin -0.6in
\advance \textwidth 1.0in 
\advance \textheight 0.5in 
\advance \topmargin -0.5in
\usepackage{hyperref}
\newtheorem{remark}{Remark}[section]

\newtheorem{theorem}{Theorem}[section]

\newtheorem{lemma}{Lemma}[section]

\begin{document}
\title{Stability Criteria for Linear Hamiltonian Systems Under Impulsive Perturbations}

\author{Z. Kayar, A. Zafer}

\address{Z. Kayar \newline
Department of Mathematics, Middle East Technical University 06531 Ankara, Turkey}
\email{zkayar@metu.edu.tr}

\address{A. Zafer \newline
Department of Mathematics, Middle East Technical University 06531 Ankara, Turkey}
\email{zafer@metu.edu.tr}

\thanks{Submitted February 9, 2010. }

\begin{abstract}
Stability  criteria are given for linear periodic Hamiltonian systems with impulse effect. A Lyapunov type inequality and a disconjugacy criterion are also established. The results improve the ones in the literature for such systems.
\end{abstract}

\maketitle

\section{Introduction\label{Sec1}}

Consider the Hamiltonian system  
\begin{equation}
y^{\prime }=JH(t)y,\quad t\in \mathbb{R},  \label{1.1}
\end{equation}%
where $H(t)$ is a symmetric $2\times 2$ matrix with entries $h_{jk}(t)$ and 
\[
J=\left[ 
\begin{array}{cc}
0 & 1 \\ 
-1 & 0%
\end{array}%
\right]. 
\]%

By setting 
$y_{1}(t)=x(t)$, $y_{2}(t)=u(t)$,
$h_{12}(t)=h_{21}(t)=a(t)$, $h_{22}(t)=b(t)$, $h_{11}(t)=c(t)$, 
one can write the system (\ref{1.1}) as 
\begin{equation}
x^{\prime }=a(t)x+b(t)u,\quad u^{\prime }=-c(t)x-a(t)u,\quad t\in \mathbb{R}.
\label{1.2}
\end{equation}
We remark that the second-order differential equation 
\begin{eqnarray*}
(p(t)x^{\prime })^{\prime }+q(t)x=0,\quad t\in \mathbb{R},  \label{1.3}
\end{eqnarray*}%
is a special case of (\ref{1.2}) via  
\[
a(t)\equiv 0,\quad b(t)=\frac{1}{p(t)},\quad c(t)=q(t), 
\]
where $p(t)$ and $q(t)$ are real-valued functions and $p(t)\neq 0$ for any $
t\in \mathbb{R}$.

In what follows we assume that 
$a(t)$, $b(t)$, and $c(t)$ satisfy the periodicity conditions
\begin{eqnarray*}
a(t+T)=a(t),\quad b(t+T)=b(t),\quad c(t+T)=c(t),\quad t\in \mathbb{R}. 
\end{eqnarray*}

The system (\ref{1.2}) (or (\ref{2.1})) is said to be \textit{stable} if all
solutions are 
bounded on $\mathbb{R}$, \textit{unstable} if all nontrivial
solutions are unbounded on $\mathbb{R}$, and \textit{conditionally stable}
if there exits a nontrivial solution bounded on $\mathbb{R}$. \label{def} 

The following well known stability theorem was given by  M. Krein in \cite{mkrein}.

\begin{theorem}
\label{t01} If
\begin{eqnarray*}
b(t)\geq 0,\quad c(t)\geq 0,\quad b(t)c(t)-a^{2}(t)\geq 0;  
\end{eqnarray*}%
\begin{eqnarray*}
\int_{0}^{T}b(t)\,\mathrm{d}t\int_{0}^{T}c(t)\,\mathrm{d}t-\bigg(%
\int_{0}^{T}a(t)\,\mathrm{d}t\bigg)^{2}>0;  
\end{eqnarray*}%
\begin{eqnarray*}
\int_{0}^{T}|a(t)|\,\mathrm{d}t+\bigg\{\int_{0}^{T}b(t)\,\mathrm{d}%
t\int_{0}^{T}c(t)\,\mathrm{d}t\bigg\}^{1/2}<2,  
\end{eqnarray*}%
then system (\ref{1.2}) is stable.
\end{theorem}

In \cite{GusKay} Guseinov and Kaymak\c calan obtained a similar result by making use of the Floquet theorem.

\begin{theorem}
\label{t02} If
\begin{eqnarray*}
b(t)>0,\quad c(t)\geq 0, \quad b(t)c(t)-a^{2}(t)\geq 0;  
\end{eqnarray*}
\begin{eqnarray*}
b(t)c(t)-a^{2}(t)\not\equiv 0; 
\end{eqnarray*}
\begin{eqnarray*}
\int_{0}^{T}|a(t)|\,\mathrm{d}t+\bigg\{\int_{0}^{T}b(t)\,\mathrm{d}%
t\int_{0}^{T}c(t)\,\mathrm{d}t\bigg\}^{1/2}<2,  
\end{eqnarray*}
then system (\ref{1.2}) is stable.
\end{theorem}

Recently, Guseinov and Zafer \cite{guszaf2} established stability criteria for Hamitonian systems under impulse effect of the form
\begin{equation}
\begin{array}{l}
x^{\prime }=a(t)x+b(t)u,\quad u^{\prime }=-c(t)x-a(t)u,\quad t\neq \tau _{i},
\\ 
x(\tau _{i}+)=\alpha _{i}x(\tau _{i}-),\quad u(\tau _{i}+)=\alpha _{i}u(\tau
_{i}-)-\beta _{i}x(\tau _{i}-),%
\end{array}
\label{2.1}
\end{equation}%
where $t\in \mathbb R$ and $i\in \mathbb Z$; $\{\tau _{i}\}$ ($i\in \mathbb Z$) is a sequence of real numbers (impulse points) such that $\tau _{i}<\tau _{i+1}$ for all $i\in \mathbb Z$, and that  
\[
\tau _{i+r}=\tau _{i}+T\;\;(i\in \mathbb Z),\quad 0<\tau _{1}<\tau _{2}<\cdots
<\tau _{r}<T
\]%
for some  positive
real number $T$ and some positive integer $r$; The functions $a,b,c:\;\mathbb R\backslash \{\tau _{i}:\,i\in \mathbb Z\}\rightarrow \mathbb R$
and the real sequences $\{\alpha _{i}\},\{\beta _{i}\}$ ($i\in \mathbb Z$) satisfy the periodicity 
conditions: 
\begin{eqnarray*}
&&
\mbox{$a(t+T)=a(t)$, $b(t+T)=b(t)$, $c(t+T)=c(t)$, $\;
t\in \mathbb R \backslash \{\tau_i:\, i\in \mathbb Z\}$}; \\
&& \mbox{$\alpha_i\neq 0$, $\alpha_{i+r}=\alpha_i$,
$\beta_{i+r}=\beta_i$,$\;i\in \mathbb Z$};\\
&&a,b,c\in \mathrm{PC}[0,T],
\end{eqnarray*}
where 
${\mathrm{PC}}[0,T]$ denotes the set of functions $$f:\,[0,T]\backslash \{\tau _{1},\tau
_{2},\ldots ,\tau _{r}\}\rightarrow \mathbb R$$ such that
$f\in C([0,T]\backslash \{\tau
_{1},\tau _{2},\ldots ,\tau _{r}\})$ and 
the left-hand limit $f(\tau _{i}-)$ and
the right-hand limit $f(\tau _{i}+)$ exist (finite) for each $i\in N_1^r \equiv \{1,2,\ldots ,r\}.$ 
As usual, by ${\mathrm{PC}}^{1}[0,T]$ we mean the set of
functions $f\in {\mathrm{PC}}[0,T]$ such that $f'\in {\mathrm{PC}}[0,T].$

The next two theorems, extracted from \cite{guszaf2}, are of particular importance for our work in this paper. We note that the inequality \eqref{113} is given as a nonstrict inequality in \cite{guszaf2}, however it should be a strict one as pointed out in \cite{wang}.

\begin{theorem} \label{t03} Assume  that 
\begin{equation}\label{111}\displaystyle
\prod_{i=1}^{r}\alpha _{i}^{2}=1;  
\end{equation}
\begin{equation}\label{111a}
b(t)>0,\quad \int_{0}^{T}\bigg(c(t)-\frac{a^{2}(t)}{b(t)}\bigg)\,\mathrm{d}t+\sum_{i=1}^{r}\frac{\beta _{i}}{\alpha _{i}}>0;  
\end{equation}%
\begin{equation}\displaystyle
\int_{0}^{T}|a(t)|\,\mathrm{d}t+\bigg[\int_{0}^{T}b(t)\,\mathrm{d}t\bigg]^\frac{1}{2}\,\bigg[\int_{0}^{T}c^{+}(t)\,\mathrm{d}t+\sum_{i=1}^{r}\bigg(\frac{
\beta _{i}}{\alpha _{i}}\bigg)^{+}\bigg]^\frac{1}{2}\,< 2,  \label{113}
\end{equation}
where \begin{equation}
c^{+}(t)=\max \{c(t),0\},\quad \left( \frac{\beta _{i}}{\alpha _{i}}\right)
^{+}=\max \left\{ \frac{\beta _{i}}{\alpha _{i}},0\right\}.  \label{3.16}
\end{equation}

Then impulsive system (\ref{2.1}) is stable.
\end{theorem}

We say that the condition (C) is satisfied if one of the following statements holds:
\begin{itemize}
	\item[] (C1) $\beta _{i}\neq 0$ $\exists i\in N_1^r=\{1,2,\ldots,r\}$,
	\item[] (C2)  $\beta _{i}= 0$ $\forall i\in N_1^r$, $\; a/b \not\in \mathrm{PC}^{1}[0,T]$,
	\item[] (C3)  $\beta _{i}= 0$ $\forall i\in N_1^r$, $\; a/b\in \mathrm{PC}^{1}[0,T]$, 
$\;\displaystyle \left( \frac{a(t)}{b(t)}\right)'-c(t)+\frac{a^{2}(t)}{b(t)}%
\not\equiv 0$. 
\end{itemize} 
\begin{theorem} \label{t04} Assume that (\ref{111}) and (\ref{113}) hold,  $a/b \in C[0,T]$, the condition (C) is satisfied, and that 
\begin{equation}\label{111b}
b(t)>0,\quad\int_{0}^{T}\bigg(c(t)-\frac{a^{2}(t)}{b(t)}\bigg)\,\mathrm{d}t+\sum_{i=1}^{r}\frac{\beta _{i}}{\alpha _{i}}=0.  
\end{equation}
Then the impulsive system (\ref{2.1}) is stable.
 
\end{theorem}

All results above have limitations. For instance, if \begin{equation}
\int_{0}^{T}|a(t)|\,\mathrm{d}t>2
\label{acond} \end{equation} 
or 
\begin{equation}
\int_{0}^{T}b(t)\,\mathrm{d}t\left[\int_{0}^{T}c^{+}(t)\,\mathrm{d}t+\sum_{i=1}^{r}\bigg(\frac{
\beta _{i}}{\alpha _{i}}\bigg)^{+}\right]> 4,  \label{acond2}
\end{equation}
then none of the above theorems is applicable, in view of the fact that  the inequality (\ref{113}) fails to hold.

Very recently, Wang \cite{wang} has proved the following stability theorem for system (\ref{1.2}) in which (\ref{acond}) is allowed. 

\begin{theorem} \label{t1c1} Assume  that  
\begin{equation}
b(t)>0,\quad \int_{0}^{T}\bigg(c(t)-\frac{a^{2}(t)}{b(t)}\bigg)\,\mathrm{d}%
t>0;  \label{3.2t1c1}
\end{equation}%
\begin{equation}\label{1.6a}
\int_0^T b(t)\mathrm{d}t\int_0^T c^{+}(t)\,\mathrm{d}t<4 \exp\left(-2\int_0^T \left|a(u)\right|\mathrm{d}u\right).
\end{equation}
Then system (\ref{1.2}) is stable.
\end{theorem}

In the present paper our aim is to extend the results in \cite{guszaf2} so as to allow the possibilities (\ref{acond}) and (\ref{acond2}).  The paper is organized as follows. In the next section we outline some basic facts about Floquet theory for impulsive equations. The main results of the paper are stated in Section \ref{Sec3}. The proof of the theorems are given in Section \ref{Sec4}. Finally, in Section \ref{Sec5} by making use of a Lyapunov type inequality we derive a disconjugacy criterion for the impulsive system (\ref{2.1}). 

\setcounter{equation}{0}
\section{Preliminaries}\label{Sec2}
The theory of impulsive differential equations has been developed very
extensively over the past 20 years, see \cite{bs,sp} and the
references cited therein. Below for completeness we provide some basic facts given in \cite{guszaf2} about impulsive system (\ref{2.1}), see \cite{sp} for more details.

Let 
\[
X(t)=\left[ 
\begin{array}{lll}
x_{1}(t) &  & x_{2}(t) \\ 
u_{1}(t) &  & u_{2}(t)%
\end{array}%
\right] ,\quad X(0)=I_2 
\]%
be a fundamental matrix solution of (\ref{2.1}). Under the above periodicity
conditions system (\ref{2.1}) becomes periodic and therefore the Floquet
theory holds. For details of the Floquet theory we refer to \cite{bs,sp}.

The Floquet multipliers (real or complex) of (\ref{2.1}) are the roots of 
\[
\det (\rho I_2-X(T))=0, 
\]%
which is equivalent to 
\begin{equation}
\rho ^{2}-A\rho +B=0,  \label{2.2}
\end{equation}%
where 
\begin{equation}
A=x_{1}(T)+u_{2}(T),\quad B=\prod_{i=1}^{r}\alpha _{i}^{2}.  \label{2.3}
\end{equation}%
It follows from the Floquet theory that corresponding to each (complex) root 
$\rho $ there is a nontrivial solution $y(t)=(x(t),u(t))$ of (\ref{2.1})
such that 
\begin{equation}
y(t+T)=\rho \,y(t),\quad t\in \mathbb R\backslash \{\tau _{i}:\;i\in \mathbb Z\}.
\label{2.4}
\end{equation}%
Note that if $\rho _{1}$ and $\rho _{2}$ are the Floquet multipliers then we
have 
\[
\rho _{1}+\rho _{2}=A,\quad \rho _{1}\rho _{2}=B. 
\]%
System (\ref{2.1}) has two linearly independent solutions and any solution
of (\ref{2.1}) can be written as their linear combination.

In view of (\ref{2.4}) we easily obtain that 
\[
y(t+kT)=\rho ^{k}\,y(t),\quad k\in \mathbb Z. 
\]%
Clearly, if $|\rho |\neq 1$ then $y(t)$ is an unbounded solution of system (%
\ref{2.1}). It follows that if $\prod_{i=1}^{r}\alpha _{i}^{2}\neq 1$ then $%
B=\rho _{1}\rho _{2}\neq 1$ and so at least one of the multipliers will have
modulus different from $1$. Therefore (\ref{2.1}) cannot
be stable unless $B=1$. 

Clearly, if $B=1$ then (\ref{2.2}) becomes 
\begin{equation}
\rho ^{2}-A\,\rho +1=0.  \label{2.6}
\end{equation}%
Since the coefficients in equation (\ref{2.1}) are real, the components of
solutions $(x_{1}(t),u_{1}(t))$ and $(x_{2}(t),u_{2}(t))$ can be taken
real-valued. Obviously, the number $A$ defined by (\ref{2.3}) then becomes
real as well.

\begin{lemma}[\cite{guszaf2}]
\label{Lem2.1} Assume that $B=1$. Then system (\ref{2.1}%
) is unstable if $|A|>2$, and stable if $|A|<2$. In case $|A|=2$, system (\ref{2.1}) is stable when $u_1(T)=x_2(T)=0$, conditionally stable and not stable otherwise.

\end{lemma}

In what follows, by a zero of a function $x(t)$ we mean a real number $t_{0}$
for which $x(t_{0}-)=0$ or $x(t_{0}+)=0$. Obviously, if $x(t)$ is continuous
at $t_{0}$ then $t_{0}$ becomes a real zero, i.e., $x(t_{0})=0$. Since $%
\alpha _{i}\neq 0$, for a solution $x(t)$ of equation (\ref{2.1}) $%
x(t_{0}-)=0$ implies $x(t_{0}+)=0$ and conversely. In case no such $t_{0}$
exists we will write $x(t)\neq 0$.

\begin{lemma}[\cite{guszaf2}]
\label{Lem3.1} Suppose that (\ref{111}) and (\ref{111a}) hold. If
\begin{equation}
A^{2}\geq 4,  \label{3.1}
\end{equation}%
then system (\ref{2.1}) has a nontrivial solution $y(t)=(x(t),u(t))$
possessing the following properties: there exist two points $t_{1}$ and $%
t_{2}$ in $\mathbb R$ such that 
\begin{equation}\label{t1t2}
0\leq t_{1}\leq T, \ t_{2}>t_{1}, \ t_{2}-t_{1}\leq T,
\end{equation} 
$x(t)$ has zeros at $t_{1}$ and $t_{2}$, and $x(t)\neq 0$
for all $t_{1}<t<t_{2}$.
\end{lemma}

\begin{lemma}[\cite{guszaf2}]
\label{Lem3.2} Suppose that (\ref{111}), (\ref{111b}), and (\ref{3.1}) hold, $a/b\in \mathrm{C}[0,T]$, and the condition (C) is satisfied.   
Then the conclusion of Lemma \ref{Lem3.1} remains valid.
\end{lemma}

\setcounter{equation}{0}
\section{Stability Criteria} \label{Sec3}

The main results of this paper are given in the following two theorems. 

\begin{theorem} \label{t1} Assume  that 
\begin{equation}
\prod_{i=1}^{r}\alpha _{i}^{2}=1;  \label{2.5t1}
\end{equation}%
\begin{equation}
b(t)>0,\quad\int_{0}^{T}\bigg(c(t)-\frac{a^{2}(t)}{b(t)}\bigg)\,\mathrm{d}%
t+\sum_{i=1}^{r}\frac{\beta _{i}}{\alpha _{i}}>0;  \label{3.2t1}
\end{equation}%
\begin{equation}\label{scr}\begin{array}{l}
\displaystyle\exp\left(2\int_0^T \left|a(u)\right|\mathrm{d}u\right)\left[\int_0^T b(t)\mathrm{d}t\right]\left[\int_0^T c^{+}(t)\,\mathrm{d}t+\sum_{i=1}^r\left(\frac{\beta _i}{\alpha _i}\right)^{+}\right]<4.
\end{array}
\end{equation}
Then impulsive system (\ref{2.1}) is stable
\end{theorem}

In case (\ref{3.2t1}) fails we have the following alternative.   
\begin{theorem} \label{t2} Assume that (\ref{2.5t1}) and (\ref{scr}) hold, $a/b \in C[0,T]$, the condition (C) is satisfied, and that   
\begin{equation}
b(t)>0,\quad\int_{0}^{T}\bigg(c(t)-\frac{a^{2}(t)}{b(t)}\bigg)\,\mathrm{d}%
t+\sum_{i=1}^{r}\frac{\beta _{i}}{\alpha _{i}}=0.  \label{3.2t2}
\end{equation}%
Then impulsive system (\ref{2.1}) is stable.
\end{theorem}

\begin{remark} \rm If there is no impulse, i.e, $\alpha_i=1$ and $\beta_i=0$ for all $i\in \mathbb{Z}$, then Theorem \ref{t1} and Theorem \ref{t2} result in \cite[Corollary 4.1, Corollary 4.2]{wang}, cf. Theorem \ref{1.1} and Theorem \ref{1.2}. 
\end{remark}

\setcounter{equation}{0}
\section{Proof of Theorem \ref{t1}}\label{Sec4}
The proof is based on the arguments developed in \cite{guszaf2,wang}. In virtue of Lemma \ref{Lem2.1} it is sufficient to show that $A^{2}<4.$
Assume on the contrary that $A^{2}\geq 4.$  By Lemma \ref{Lem3.1} there exists a solution 
$y(t)=(x(t),u(t))$ with two zeros $t_{1},t_{2}\in \lbrack 0,T]$ ($%
t_{1}<t_{2}\leq t_{1}+T$) of $x(t)$ such that $x(t)\neq 0$ for all $t\in
(t_{1},t_{2})$.

Define 
\begin{equation}
z(t)=\frac{1}{\alpha _{1}\alpha _{2}\cdots \alpha _{i}}x(t),\quad v(t)=\frac{%
1}{\alpha _{1}\alpha _{2}\cdots \alpha _{i}}u(t)  \label{3.18}
\end{equation}%
for $t\in (\tau _{i},\tau _{i+1})$ and $i\in N_0^r$, where we
put again $\tau _{0}=0$, $\tau _{r+1}=T$, and make a convention that $\alpha
_{1}\alpha _{2}\cdots \alpha _{i}=1$ if $i=0$.

It is easy to verify that%
\begin{equation}
\begin{array}{l}
z^{\prime }=a(t)z+b(t)v,\quad v^{\prime }=-c(t)z-a(t)v,\quad t\neq \tau _{i},
\\ 
z(\tau _{i}+)=z(\tau _{i}-),\quad v(\tau _{i}+)=v(\tau _{i}-)-\frac{\beta
_{i}}{\alpha _{i}}z(\tau _{i}-),%
\end{array}
\label{3.19}
\end{equation}%
where $t\in [0,T]$ and $i\in \{1,2,\ldots ,r\}$.

We may define $z(\tau _{i})=z(\tau _{i}-)$ so as to make $z(t)$ continuous
on $[0,T]$. Moreover, $z^{\prime }\in \mathrm{PC}[0,T],$ $%
z(t_{1})=z(t_{2})=0,$ and $z(t)\neq 0$ for all $t\in (t_{1},t_{2})$. We may
assume without loss of generality that $z(t)>0$ on $(t_{1},t_{2})$.

From $z^{\prime }=a(t)z+b(t)v$ we see that  
\begin{equation}\label{3.90}
\displaystyle
\left[z(t)\exp\left(-\int_{t_1}^{t}a(u)\mathrm{d}u\right)\right]^{\prime}=b(t)v(t)\exp\left(-\int_{t_1}^{t}a(u)\mathrm{d}u\right)
\end{equation}
and
\begin{equation}\label{3.89}
\displaystyle
\left[z(t)\exp\left(-\int_{t}^{t_2}a(u)\mathrm{d}u\right)\right]^{\prime}=b(t)v(t)\exp\left(-\int_{t}^{t_2}a(u)\mathrm{d}u\right).
\end{equation}                                                                                                                           
Let $t_0$ be a point in $(t_1,t_2)$ such that 
\[
z(t_{0})=\max \{z(t):\;t\in (t_{1},t_{2})\}. 
\]%

Integrating (\ref{3.90}) from $t_1$ to $t_0$, we get

\begin{eqnarray*}
z(t_0)=\int_{t_1}^{t_0}b(t)v(t)\exp\left(-\int_{t_0}^{t}a(u)\mathrm{d}u\right)\mathrm{d}t,
\end{eqnarray*}
and so 
\begin{equation}\label{3.86}
\displaystyle
z(t_0)\leq\int_{t_1}^{t_0}b(t)\left|v(t)\right|\exp\left(-\int_{t_0}^{t}a(u)\mathrm{d}u\right)\mathrm{d}t.
\end{equation}
Similarly, if we integrate (\ref{3.89}) from $t_0$ to $t_2$ then we obtain

\begin{equation}\label{3.85}
\displaystyle
z(t_0)\leq\int_{t_0}^{t_2}b(t)\left|v(t)\right|\exp\left(-\int_{t_0}^{t}a(u)\mathrm{d}u\right)\mathrm{d}t
\end{equation}

From (\ref{3.86}) and (\ref{3.85}),  
\begin{equation}\label{3.84}
\displaystyle
2z(t_0)\leq\int_{t_1}^{t_2}b(t)\left|v(t)\right|\exp\left(-\int_{t_0}^{t}a(u)\mathrm{d}u\right)\mathrm{d}t.
\end{equation}
Using the Cauchy-Schwarz inequality in (\ref{3.84}) leads to 
\begin{eqnarray} \nonumber 
&&4z^2(t_0)\leq \bigg[\int_{t_1}^{t_2}b(t)|v(t)|\exp\bigg(-\int_{t_0}^{t}a(u)\mathrm{d}u\bigg)\mathrm{d}t\bigg]^2\\
&&\qquad\leq\bigg[\int_{t_1}^{t_2}b(t)\exp\left(-2\int_{t_0}^{t}a(u)\mathrm{d}u\right)\mathrm{d}t\bigg]\bigg[\int_{t_1}^{t_2}b(t)v^2(t)\mathrm{d}t\bigg]
\label{in1}
\end{eqnarray}

On the other hand, in view of (\ref{3.19}), by a direct calculation we have 
\begin{eqnarray} \label{vz1}
(vz)^{\prime }=-c(t)z^{2}+b(t)v^{2},\quad t\neq \tau _{i}
\end{eqnarray}%
and 
\begin{eqnarray}\label{vz2}
(vz)(\tau _{i}+)-(vz)(\tau _{i}-)=\displaystyle-\frac{\beta _{i}}{\alpha _{i}}z^{2}(\tau
_{i}).%
\end{eqnarray}%
Integrating (\ref{vz1}) from $t_1$ to $t_2$  and using (\ref{vz2}) we obtain 
\begin{eqnarray*}
\sum_{t_{1}\leq \tau _{i}<t_{2}}\frac{\beta _{i}}{\alpha _{i}}z^{2}(\tau
_{i})=\int_{t_{1}}^{t_{2}}\bigg[b(t)v^{2}(t)-c(t)z^{2}(t)\bigg]\,\mathrm{d}t,
\end{eqnarray*}%
and hence   
\begin{equation}
\int_{t_{1}}^{t_{2}}b(t)v^{2}(t)\,\mathrm{d}t\leq
\int_{t_{1}}^{t_{2}}c^{+}(t)z^{2}(t)\,\mathrm{d}t+\sum_{t_{1}\leq \tau
_{i}<t_{2}}\bigg(\frac{\beta _{i}}{\alpha _{i}}\bigg)^{+}z^{2}(\tau _{i}).
\label{in2}
\end{equation}

From (\ref{in1}) and (\ref{in2}) we easily get 
\begin{eqnarray}\nonumber
4z^2(t_0)&&\leq\left[\int_{t_1}^{t_2}b(t)\exp\left(-2\int_{t_0}^{t}a(u)\mathrm{d}u\right)\mathrm{d}t\right]\times \\
&&\nonumber\qquad\qquad\bigg[\int_{t_{1}}^{t_{2}}c^{+}(t)z^{2}(t)\,\mathrm{d}t+\sum_{t_1\leq \tau_i<t_2}\left(\frac{\beta _i}{\alpha _i}\right)^{+}z^{2}(\tau_i)\bigg]\\
&&\nonumber \leq z^2(t_0) \left[\int_{t_1}^{t_2}b(t)\exp\left(-2\int_{t_0}^{t}a(u)\mathrm{d}u\right)\mathrm{d}t\right]\times \\
&&\qquad\qquad
\bigg[\int_{t_{1}}^{t_{2}}c^{+}(t)\,\mathrm{d}t+\sum_{t_1\leq \tau_i<t_2}\left(\frac{\beta _i}{\alpha _i}\right)^{+}\bigg]
\label{lyp}
\end{eqnarray}
where we have used the fact that 
$$\mbox{$z(t_{0})\geq z(t)$  for all $t\in (t_{1},t_{2})$}.$$

From (\ref{lyp}) we obtain  
\begin{eqnarray*}\label{lyap}
4&\leq&\left[\int_{t_1}^{t_1+T}b(t)\exp\left(-2\int_{t_0}^{t}a(u)\mathrm{d}u\right)\mathrm{d}t\right]
\left[\int_{t_{1}}^{t_{1}+T}c^{+}(t)\,\mathrm{d}t+\sum_{t_1\leq \tau
_i<t_1+T}\left(\frac{\beta _i}{\alpha _i}\right)^{+}\right]\\
&=&\left[\int_0^T b(t+t_1)\exp\left(-2\int_{t_{0}-t_1}^{t} a(u+t_1)\mathrm{d}u\right)\mathrm{d}t\right]
\left[\int_0^T c^{+}(t)\,\mathrm{d}t+\sum_{i=1}^r\left(\frac{\beta _i}{\alpha _i}\right)^{+}\right]\\
&\leq & \exp\left(2\int_{0}^{T} |a(u+t_1)|\mathrm{d}u\right)\left[\int_0^T b(t+t_1)\mathrm{d}t\right]
\left[\int_0^T c^{+}(t)\,\mathrm{d}t+\sum_{i=1}^r\left(\frac{\beta _i}{\alpha _i}\right)^{+}\right]\\
&= & \exp\left(2\int_{0}^{T} |a(t)|\mathrm{d}t\right)\left[\int_0^T b(t)\mathrm{d}t\right]
\left[\int_0^T c^{+}(t)\,\mathrm{d}t+\sum_{i=1}^r\left(\frac{\beta _i}{\alpha _i}\right)^{+}\right].
\end{eqnarray*} 
This inequality clearly contradicts (\ref{scr}).
\qed%

The proof of Theorem \ref{t2} is exactly the same as that of Theorem \ref{t1}, except that Lemma \ref{Lem3.2} is used instead of Lemma \ref{Lem3.1}.

\section{A Disconjugacy Criterion} \label{Sec5}

In this section we establish a disconjugacy criterion for system (\ref{2.1}). 
The system is called disconjugate on the interval $[t_1,t_2]$ if and only if there is no real solution 
$(x,u)$ of (\ref{2.1}) with $x$ nontrivial and having two or more zeros on $[t_1,t_2]$ in the sense described at the beginning of Section \ref{Sec4}. 
The periodicity conditions need not hold. That is, the system (\ref{2.1}) is not required to be periodic. So we let \textrm{PC}$[t_{1},t_{2}]$
denote the set of real-valued functions
\[
f:[t_{1},t_{2}]\setminus \{\tau _{i}:i\in \mathbb Z\}\rightarrow \mathbb R
\]%
which are continuous on \textbf{\ }$[t_{1},t_{2}]\setminus \{\tau _{i}:i\in 
\mathbf{Z\}}$ and for which the left limit $f(\tau _{i}-)$ and the right
limit $f(\tau _{i}+)$ exist (finite) for each $i\in \mathbf{Z.}$

The proof of Theorem \label{t4} is based on a Lyapunov type inequality related to the impulsive system (\ref{2.1}). We note that Lyapunov inequalities are useful in oscillation, disconjugacy, and boundary value problems. 

\begin{theorem}[Lyapunov type inequality] \label{t3}
 Let $a,b,c\in $\textrm{PC}$[t_{1},t_{2}],$ $b(t)>0,$
and $\alpha _{i}\neq 0$ for all $i\in \mathbf{Z}.$ Suppose that (\ref{2.1})
has a solution $(x(t),u(t))$ such that $x(t_1+)=x(t_2-)=0$ and $x(t)\neq 0$ on $(t_{1},t_{2}).$ 
Then the inequality 
\begin{equation}\label{lyapunovine}
\left[\int_{t_1}^{t_2}b(t)\exp\left(-2\int_{t_0}^{t}a(u)\mathrm{d}u\right)\mathrm{d}t\right]
\left[\int_{t_{1}}^{t_{2}}c^{+}(t)\,\mathrm{d}t+\sum_{t_1\leq \tau
_i<t_2}\left(\frac{\beta _i}{\alpha _i}\right)^{+}\right]\geq 4
\end{equation}
holds for some $t_0\in(t_1,t_2)$.
\end{theorem}
\begin{proof}  
Taking  $y(t)=(x(t),u(t))$ with $t_{1},t_{2}$ ($
t_{1}<t_{2}$) the two zeros of $x(t)$ such that $x(t)\neq 0$ for all $t\in
(t_{1},t_{2})$, and proceeding as in the proof of Theorem \ref{t1} we arrive at (\ref{lyp}), which is the same as (\ref{lyapunovine}). \end{proof}

\begin{theorem}[Disconjugacy] \label{t4a}
 Let $a,b,c\in \textrm{PC}[t_1,t_2]$, $b(t)>0,$
and $\alpha_{i}\neq 0$ for all $i\in \mathbf{Z}.$ If for every $t_0\in(t_1,t_2)$,
\begin{equation}
\hspace{-.71cm}\left[\int_{t_1}^{t_2} b(t)\exp\left(-2\int_{t_0}^{t}a(u)\mathrm{d}u\right)\mathrm{d}t\right]
\left[\int_{t_1}^{t_2} c^{+}(t)\,\mathrm{d}t+\sum_{t_{1}\leq \tau _{i}<t_{2}}\left(\frac{\beta _i}{\alpha _i}\right)^{+}\right]< 4,
\label{ly-d}
\end{equation}
then system (\ref{2.1}) is disconjugate on $[t_1,t_2]$.
\end{theorem}
\begin{proof}  
Suppose on the contrary that there is a solution
$y(t)=(x(t),u(t))$ with nontrivial $x(t)$ having two zeros 
$s_1,s_2\in [t_1,t_2]$ 
($s_{1}<s_{2}$) such that $x(t)\neq 0$ for all $t\in (s_{1},s_{2})$. Applying Theorem \ref{t3} we see that there is a $t_0\in(s_{1},s_{2})$  
$$
4\leq \left[\int_{s_1}^{s_2} b(t)\exp\left(-2\int_{t_0}^{t}a(u)\mathrm{d}u\right)\mathrm{d}t\right]
\left[\int_{s_1}^{s_2} c^{+}(t)\,\mathrm{d}t+\sum_{s_{1}\leq \tau _{i}<s_{2}}\left(\frac{\beta _i}{\alpha _i}\right)^{+}\right]$$
$$
\quad\leq\left[\int_{t_1}^{t_2} b(t)\exp\left(-2\int_{t_0}^{t}a(u)\mathrm{d}u\right)\mathrm{d}t\right]
\left[\int_{t_1}^{t_2} c^{+}(t)\,\mathrm{d}t+\sum_{t_{1}\leq \tau _{i}<t_{2}}\left(\frac{\beta _i}{\alpha _i}\right)^{+}\right]
$$

Clearly, this inequality contradicts (\ref{ly-d}). The proof is complete. \end{proof}

\end{document}